\documentclass[reqno,11pt]{amsart}
\usepackage{eucal}
\usepackage{amssymb,amscd,amsmath}
\usepackage{verbatim}
\usepackage{color}
\usepackage{latexsym}
\usepackage{mathrsfs}
\usepackage{bbm}
\usepackage{array}
\usepackage[colorlinks, linkcolor=black, citecolor=blue, linktocpage, urlcolor=blue]{hyperref}

\usepackage{graphicx,wrapfig}
\usepackage{calc}
\usepackage{enumitem}
\usepackage{tensor}
\usepackage{etoolbox}
\usepackage{soul}

\newcolumntype{M}[1]{>{\centering\arraybackslash}m{#1}}

\numberwithin{equation}{section}

\newtheoremstyle{plainstyle} 
  {}                          
  {}                          
  {\normalfont}               
  {}                          
  {\bfseries}                 
  {.}                         
  { }                         
  {}                          

\theoremstyle{plainstyle}

\newtheorem{theorem}{Theorem}[section]
\newtheorem{conjecture}[theorem]{Conjecture}

\newtheorem{definition}[theorem]{Definition} 
\newtheorem{corollary}[theorem]{Corollary} 
\newtheorem{proposition}[theorem]{Proposition} 
\newtheorem{lemma}[theorem]{Lemma}

\newtheorem{remark}[theorem]{Remark}

\newtheorem{question}[theorem]{Question}

\usepackage{fancyhdr}
\title{Weakly special manifolds with no rational curves}

\author{Kyle Broder}
\address{The University of Queensland,  St. Lucia,  QLD 4067, Australia}
\email{k.broder@uq.edu.au}

\author{Frederic Campana}
\address{Institut Elie Cartan, Universit\'e Henri Poincar\'e, BP 239, F54506. Vandoeuvre-les-Nancy C\'edex, et: Institut Universitaire de France}
\email{frederic.campana@univ-lorraine.fr}

\thanks{The first named author was funded as a postdoc by the Australian Government through the Australian Research Council's Discovery Projects funding scheme (project DP220102530). }

\begin{document}

\maketitle

\begin{abstract}
 Assuming the abundance conjecture and the existence of a Zariski dense set of rational curves on terminal Calabi--Yau varieties, we show that a complex projective weakly special manifold $X$ with no rational curves is an \'etale quotient of an Abelian variety. The same conclusion holds true if $X$ contains a Zariski dense entire curve, assuming Lang's conjecture. This implies that any non-hyperbolic complex projective manifold contains the image of an Abelian variety, according to another conjecture of Lang. We illustrate this last conjecture by producing examples of canonically polarised submanifolds of abelian varieties containing no subvariety of general type, except for a finite number of disjoint copies of some simple abelian variety, which can be chosen arbitrarily. We also show, more generally, that any projective manifold containing a Zariski dense entire curve appears as the `exceptional set' in Lang's sense of some general type manifold.  
\end{abstract}

\section{Definitions, Terminology}

The base field is $\Bbb C$. 

\begin{definition}
A projective manifold (resp. variety) $X$ is \emph{weakly special} if no finite \'etale cover of $X$ admits a rational fibration onto a positive-dimensional variety of general type (resp. if some/any of its smooth models is special) .
\end{definition}

We shall use some more terminology.

Recall that an \textit{entire curve} (resp. \textit{rational curve}) in a connected complex space $X$ is a non-constant holomorphic map $\mathbf{C} \longrightarrow X$ (resp. $\Bbb P_1 \longrightarrow X$) , and that $X$ is (Brody) hyperbolic if it does not contain any entire curve.

A `terminal model' is a connected normal projective variety $X$ with $K_X$ $\Bbb Q$-Cartier, having only terminal singularities and semi-ample canonical divisor. These appear in the Minimal Model Program (exposed in \cite{KMM87} and \cite{KM}, 2.34, \S3.7 and \S5 in particular).

A normal variety $X$ is a `quotient of an Abelian variety' $X$, if there is a finite surjective map $q:A\to X$ for some Abelian variety $A$. It is an \'etale quotient if $q$ is moreover \'etale: this is equivalent to $X$ smooth if $q$ is \'etale in codimension $1$.  

A fibration is a  holomorphic, surjective map between connected normal projective spaces with connected fibres.

An Abelian group scheme $f':X'\to Z'$ is a submersive fibration with all fibres Abelian varieties, admitting a holomorphic section.

\section{Statements}

 We will show that as a consequence of the abundance conjecture (Conjecture~\ref{abund}) and the existence of a Zariski dense set of rational curves on Calabi--Yau and hyperk\"ahler terminal models (Conjecture~\ref{CY+}), the following equivalences hold: 

\begin{theorem}\label{main} Assume Conjecture~\ref{abund} and Conjecture~\ref{CY+}. Then: for a smooth projective  variety $X$, the following are equivalent: \begin{itemize}
    \item[(I)] $X$ is weakly special and contains no rational curves.
   \item[(II)] $X$ is an \'etale quotient of an Abelian variety.
   
    If we also admit Lang's conjecture [L] according to which varieties of general type do not contain Zariski dense entire curves, the above conditions are also equivalent to:
    
    \item[(III)] $X$ contains a Zariski dense entire curve and no rational curve.  
    
\end{itemize}
\end{theorem}

The proof rests on an unconditional result, combining those of \cite{hoer} and \cite{Kol}:

\begin{theorem}\label{Main} Let $f:X\to Z$ be the Moishezon-Iitaka fibration of $X$ smooth projective, to $Z$ normal. Assume that its generic fibres are bimeromorphic to smooth \'etale quotients of Abelian varieties.
 If no fibre of $f$ contains any rational curve, there exists a finite \'etale cover $u:X'\to X$, an `Abelian group scheme' $f':X'\to Z'$, and a finite map: $v:Z'\to Z$ with $v\circ f'=f\circ u:X'\to Z$. 
\end{theorem}

The referee points out that Theorem 1.1 in \cite{KLSV} establishes a related result for $Z$ a curve, and generic fibres $K$-trivial varieties with canonical singularities. 

Since the abundance conjecture is known in dimension $3$, and projective $K3$ surfaces contain rational curves, we get unconditionally from the proof of \ref{main} the following result, well-known to experts (see \cite{D'}):

\begin{corollary} A smooth projective threefold not containing any rational curve is either canonically polarised, or Calabi-Yau, or after a finite \'etale cover, an Abelian group scheme over a canonically polarised base, isotrivial if the fibres are elliptic curves. Since Calabi-Yau threefolds are conjectured to contain rational curves, the second case should not occur.
\end{corollary}

\noindent The `special' varieties\footnote{$X$ projective smooth is `special' if $\kappa(X,L)<p, \forall p>0$, for any rank-one coherent subsheaf $L$ of $\Omega^p_X$.} defined in \cite{Ca04} are weakly special, and the notions are equivalent for curves and surfaces. This equivalence fails from dimension $3$ on \cite{BT,CP}. 

From \ref{main} we also get:

\begin{corollary} Assuming Conjecture~\ref{abund} and Conjecture~\ref{CY+}, a weakly special variety without rational curve is special. 
\end{corollary}

 In \cite{Ca04}, it is conjectured that a projective variety $X$ is special if and only if it contains  a Zariski dense entire curve. The threefolds constructed in \cite{CP} are weakly special but not special, and are shown not to contain Zariski dense entire curves. They must contain rational curves.

\noindent The conjectural statements assumed in Theorem~\ref{main} are:

\begin{conjecture}\label{abund} 
If a projective manifold $X$ is not uniruled, it admits a birational terminal model $X'$. Further, $\kappa(X)=\kappa(X')$ and if $K_X$ is nef, then $X \simeq X'$.
\end{conjecture}

\noindent The terminal model is obtained from a sequence of divisorial contractions and flips \cite{KM}.  Conjecture~\ref{abund} is known for projective manifolds $X$ with $\dim_{\mathbf{C}} X \leq 3$ (see \cite{Miyaoka88, Kawamata92a, Kwc92}).

\begin{conjecture}\label{CY+}
If a terminal model $X'$ has torsion canonical bundle, and is not a quotient of an Abelian variety $A$, the union of its rational curves is Zariski dense in $X'$. 
\end{conjecture}

Notice that a general\footnote{We use since 1980 the term `general' (resp. `generic') to mean outside a countable (resp. finite) union of strict subvarieties. No need thus to add `very', introduced in 1995.} Kummer variety $A/\pm 1$ of dimension $3$ or more is a terminal model, but does not contain any rational curve (\cite{Pir}, Theorem 2), and that Kummer surfaces, which contain a Zariski dense set of rational curves, are not terminal models.

 By the singular version of the Bogomolov--Beauville decomposition theorem \cite{hp, CampanaCY}, and the references there, this reduces to the case of singular Calabi--Yau and hyperk\"ahler terminal models\footnote{Excluding quotients of Abelian varieties.}.  Conjecture~\ref{CY+} is known for K3 surfaces, many hyperk\"ahler and Calabi-Yau manifolds. Gromov--Witten invariants make it plausible for the general quintics in $\mathbf{P}^4$ (see, e.g., \cite{CdGP,Morrison}), and more generally by the Kawamata-Morrison cone conjecture. See \cite{D'} for a discussion and the treatment of some special cases.

\medskip

\noindent A simpler version of the techniques used to establish Theorem \ref{main} show the following, conjectured by Lang (\cite[Conjecture 5.7, p. 123]{Lang}):

\begin{theorem}\label{hyp} Assume Conjecture~\ref{abund}, conjecture \ref{CY+} and conjecture [L]\footnote{Recall again that Conjecture [L] says that a variety of general type does not contain any Zariski dense entire curve.} in \ref{main}.  If a projective manifold $X$ does not contain any rational curves or images of  simple Abelian varieties, then $X$ is hyperbolic and canonically polarised. 
\end{theorem}

\medskip

{\bf Exceptional loci.}\label{exc set} More generally, Lang introduced the notion of `exceptional locus' $W$ of any projective manifold $X$, as the Zariski closure of the union of all entire curves contained in $X$, and conjectured that $W\subsetneq X$ if (and only if) $X$ is of general type. The subvariety $W$ if nonempty is thus the obstruction to the hyperbolicity of $X$. It obviously contains all Zariski closures of entire curves of $X$. Theorem \ref{main} above thus explains why only Abelian varieties and rational curves should appear as minimal obstructions to the hyperbolicity of complex projective manifolds.

Since it is conjectured in \cite{Ca04} that the Zariski closures of entire curves are special, and that special varieties contain Zariski dense entire curves, it follows that $W$ should also be the Zariski closure of the union of all `special' subvarieties of $X$. A trivial example of such a $W$ is the blow-up of a smooth hyperbolic manifold $Y$ of dimension at least $2$ along a submanifold $S$: $W$ is then the exceptional divisor $E$, a projective bundle over $S$. The maximal special subvarieties are the fibres of $E$ over $S$ (which is hyperbolic).

We will conclude by producing examples illustrating Theorem~\ref{hyp}, in which $W$ does not contain rational, nor elliptic curves.

\begin{theorem}\label{ex}
For each simple Abelian variety $V\subset \Bbb P_n$, and each simple Abelian variety $A$ of dimension $n+1$, there exist canonically polarised submanifolds $X\subset A\times V$ with the following properties:

1. All subvarieties of $X$ are of general type, except for a smooth divisor  $W$, a finite union of disjoint copies of $V$.  

2. Each entire curve in $X$ is contained (and Zariski dense\footnote{ If the simple Abelian variety $V$ does not contain any strict real subtorus,  the Zariski dense entire curve is even dense in this component for the metric topology.}) in some component of $W$. 

3. The Kobayashi pseudodistance of $X$ is a distance on $X\setminus W$, and vanishes on $W$.

4. $X$ has nef and big cotangent bundle, and supports a K\"ahler metric with nonpositive bisectional curvature and a pseudoconvex Finsler metric of quasi-negative holomorphic sectional curvature.
\end{theorem}

Notice that if $\dim(V)\geq 2$, then $X$ does not contain rational, or elliptic curves, and each entire curve on $X$ is transcendental. We are not aware of such examples in the previous literature.

The construction used in the proof of Theorem \ref{ex} also gives, more generally, the following examples:

\begin{theorem}\label{ex'} Let $V\subset \Bbb P_n$ be a projective manifold containing a Zariski dense entire curve. Let $A$ be a simple abelian variety of dimension $n+1$. There exists in $B:=A\times V$ a submanifold $X$ of general type such that each entire curve of $X$ is contained  in a smooth divisor $W$, finite union of disjoint copies of $V$.  The Kobayashi pseudodistance on $X$ is a distance on $X\setminus W$, conjecturally vanishes on $W$ (see Remark~\ref{d}), and $W$ is the `exceptional locus' of $X$ (in the sense of Lang). 
\end{theorem}

Theorem~\ref{ex'} shows that any projective manifold $V$ carrying a Zariski dense entire curve appears as the exceptional set of some `generically hyperbolic' projective manifold of general type. Recall that these manifolds $V$ are conjecturally characterised in algebro-geometric terms as the `special' manifolds in \cite{Ca04} to which we refer for details.

The examples in \ref{ex}, \ref{ex'}, are also related to questions posed by Diverio \cite{DiverioSurvey} in relation to the Wu--Yau theorem \cite{WuYau,TosattiYang}.  Recall that the Wu--Yau theorem states that a compact K\"ahler manifold with a K\"ahler metric of negative holomorphic sectional curvature is canonically polarised.  Diverio--Trapani \cite{DT} showed that this is still  true for quasi-negative holomorphic sectional curvature. We do not know whether the examples constructed in Theorem~\ref{ex} support K\"ahler (or even Hermitian) metrics of quasi-negative holomorphic sectional curvature (see Question~\ref{QN}).

{\bf Aknowledgements:} We thank the referee for indicating the references \cite{D'}, \cite{KLSV}, the example in Remark \ref{rem ex}, and for asking clarification of the r\^ole of Conjectures \ref{CY+} and \ref{abund} in the proof of Theorem \ref{main}, leading us to Theorem \ref{Main}, which simplifies and strengthens our previous approach. We thank Valery Alexeev for Claim 2, which plays a crucial r\^ole in the proof of Theorem \ref{Main}, and Klaus Hulek for the reference \cite{EDS}.

\section{Proof of Theorems \ref{Main} and \ref{main}.}

\begin{proof} (of Theorem \ref{Main}).
It is proved in \cite{hoer} for projective manifolds with nef cotangent bundle (which do not contain rational curves). We adapt his proof here. For this, it suffices to show that over the discriminant locus $\text{disc}(f) \subset Z^{\text{reg}}$ of $f$,  the fibers are multiples of \'etale Abelian quotients.  Once this has been established,  the proof is similar to the one of \cite{hoer}, which relies on three highly non-trivial ingredients: \cite {kaw'} for the equidimensionality of the fibres, \cite[Theorem 6.3]{Kol} which rests on the finiteness of the local monodromies around $\text{disc}(f)$ and Deligne's results on the variations of Hodge structures \cite{Del}.

The assumptions of Theorem \ref{Main} thus say that the generic smooth fibres of $f$ are \'etale quotients of Abelian varieties. We next show that over the discriminant locus $\text{disc}(f) \subset Z^{\text{reg}}$ of $f$,  the fibers are multiples of \'etale quotients of Abelian varieties. For this we may restrict to generic one-dimensional complex disc in $Z^{reg}$ meeting transversally $\text{disc}(f)$ in one  point, and thus need to show the following:

{\bf Claim 1:} If $g:Y\longrightarrow \mathbf{D}$ is a projective fibration from the smooth $Y$ over the unit disk $\mathbf{D}$, with all of its fibers outside of the central one \'etale quotients of Abelian varieties, if $Y$ does not contain rational curves, then the central fiber is a multiple of a smooth manifold which is an \'etale quotient of an Abelian variety.

\begin{proof} (of the Claim:) let $Y_0:=g^*(0)$ be the central fibre of $g$, and $R\subset Y_{0,red}$ be the singular set of the reduction of $Y_0$: it has codimension at least $1$ in $Y_0$, hence at least $2$ in $Y$. We thus have a natural isomorphism $\pi_1(Y)\simeq \pi_1(Y\setminus R)$ since $Y$ is smooth. 

From Step 1 of the proof of \cite{hoer}, Lemma 2.2, we know that $Y_{0,red}$ is irreducible. Thus $g^*(0)=s.(Y_{0,red})$, $s>0$ the multiplicity of $Y_0$.

Let $r:D'\to D, r':D''\to D'$ be cyclic covers of degree $s, k$ of $D,D'$ branched over the single point $0$, $D',D''$ being discs. Let $Y',Y''$ be the normalisations of the fibre product $Y\times_DD'$, $Y''=Y'\times D'D''$, equipped with their natural maps: $u:Y'\to Y, u':Y''\to Y$, finite, and $g':Y'\to D', g'':Y''\to D''$. We next choose $k$ so that the fibres $Y''_{z''}$ of $g''$ over $0\neq z''\in D''$ are smooth Abelian varieties (\'etale covers of $g^{-1}(z):=Y_z,z=(r\circ r'')(z'')\in D$). By the choice of $s$, $Y'_0:=(g')^{-1}(0)$ is reduced, hence smooth ouside $Q:=g^{-1}(R)$ since $u:(Y'\setminus Q)\to (Y\setminus R)$ is \'etale, corresponding to a finite subgroup $G$ of finite index $s$ inside $\pi_1(Y\setminus R)$. From the isomorphism $\pi_1(Y)\simeq \pi_1(Y\setminus R)$, we deduce a finite \'etale cover $Y^{(s)}\to Y$ of degree $s$ which coincides with $Y'$ over $Y\setminus R$, from which we get the equality $Y^{(s)}=Y'$ since both are normal. In particular, $Y'$ is smooth, and $u:Y'\to Y$ is \'etale everywhere. 

A local computation (valid in any dimension, see for example \cite{BPV}, III.9.1,9.2) shows that over $Y\setminus R$, the map $u'$ is \'etale outside of the central fibres $Y''_0$ and $Y'_0$, and totally ramified of degree $k$ over the central fibres. By the lower semi-continuity of the number cardinality of fibres of a finite map,  $u':Y''\to Y'$ ramifies at order $k$ over $Y''_0:=(g'')^{-1}(0)$, is \'etale ouside of the central fibres, and $u':Y''_{0,red}\to Y'_{0,red}$ is set-theoretically isomorphic. 

We now use the following crucial result and argument, communicated to us by Valery Alexeev (\cite{A}):

{\bf Claim 2:} (V. Alexeev) There exists a fibration $g^*:Y^*\to D''$ which coincides with $g'':Y''\to D"$ over $(D'\setminus \{0\})$, which is bimeromorphic to $g'':Y''\to D''$ over $D''$, and such that its central fibre $Y^*_0$ is a semi-toric variety. Moreover, $Y^*$ has only abelian quotient singularities. If $Y^*$ is not uniruled, then $g^*:Y^*\to D''$ is a smooth fibration with all its fibres Abelian varieties.

Proof of Claim 2 (after V. Alexeev: \cite{A}): it is a consequence of a construction by Mumford (see \cite{Mu},\cite{AN}, a detailed exposition is given in \cite{EDS}, especially section 2.5 about the toroidal compactifications of Siegel spaces): each component of $Y^*_0$ is a variety on which $G_0$, a semi-Abelian group variety (extension of an abelian variety by some $(\Bbb C^*)^r$) acts with an open orbit. The singularities of $Y^*$ are quotient abelian by considering the subdivisions of a fan in Mumford's construction. If the toric rank $r=r(G_0)$ is positive, these components are uniruled. Hence Claim 2.

\medskip

We can now conclude the proof of Claim 1: Lemma \ref{Phi} below applied to $X'=Y^*$ and $X=Y''$ over $Z=D''$ shows that $\phi:Y^*\to Y''$ is regular. Since $Y''_0$ were uniruled if $Y_0^*$ were uniruled, we see that $Y^*$ does not contain rational curves in the fibres of $g^*$, and so is a smooth fibration isomorphic to $Y''$, which is thus smooth. 

Thus $u:Y''_{0,red}=Y'_{0,red}\to Y $is \'etale, and since all the fibres of $g''$ are Abelian varieties, $Y_0$ is an \'etale undercover of the Abelian variety $Y'_0$. Claim 1 is established.\end{proof}

\begin{lemma}\label{Phi} Let $f':X'\to Z'$ and $f:X\to Z'$ be fibrations with $X', X$ projective normal. Assume that there exists a birational map $\Phi:X'\dasharrow X$ such that $f\circ \Phi=\Phi'$, and that $X$ does not contain any rational curve in the fibres of $f$. 

1. If the singularities of $X'$ are quotient\footnote{Log-terminal singularities suffice, in fact.}, then $\Phi:X'\to X$ is regular (i.e: holomorphic). 

2. If $X$ has quotient singularities and $X'$ does not contain rational curves, then $\Phi$ is isomorphic. 

3. Taking $Z'$ a point, if $X$ does not contain any rational curve, and if $X'$ is a terminal quotient of an Abelian variety, then $X$ and $X'$ are isomorphic. 
\end{lemma}

\begin{proof}1. The singularities of $X'$ being quotient, if $X"$ is the normalisation of the graph of $\Phi$ inside $X'\times X$, its fibres over $X'$ are chain-rationally connected. They thus project to points in $X$. Thus $X"=X'$ by normality of $X'$, and $\Phi$ is the projection of $X"=X$ onto $X'$.

2. Apply Claim 1 to $\Phi^{-1}$.

3. The singularities of $X'$ are quotient since there is $q:A\to X'$ finite, surjective, \'etale in codimension $1$ (\cite{Ca FST}, \S5), and $\Phi: X'\to X$ is thus regular. We need to show that it is isomorphic.  Given that $q: A\to X'$ is a finite undercover of an Abelian variety $A$,  there are no subvarieties of $X'$ that contract to a point in $X$ by $\Phi$ (indeed: their inverse images in $A$  would move under the translations of $A$, and fibres of a birational map do not move ouside of the exceptional locus $E\subset A$ of $\Phi\circ q:A\to X$: if $C\subset A$ were a curve mapped to a point, if $(t+A)\subsetneq E, t$ a translation in $A$, then  its image in $X$ were a curve and so $0=(\Phi\circ q)_*(C)=(\Phi\circ q)_*(t+C)\neq 0$, contradiction). The map $\Phi$ is thus biregular, proving the claim.  
\end{proof}

We now briefly describe the steps of the proof of \cite{hoer}, Lemma 2.2, which concludes the proof of Theorem \ref{Main}:
By \cite{kaw'}, the fibration $f:X\to Z$ has equidimensional fibres. Let $S\subset Z$ be set over which the fibres do not have a smooth reduction, its codimension in $Z$ is $2$ or more, and let $T=f^{-1}(S)\subset X$. Then $\pi_1(X\setminus T)=\pi_1(X)$, from which follows that the local monodromies of $f$ around $disc(f)$ are finite.  From this follows, by \cite{Kol}, that there exists a finite \'etale cover $u:X'\to X$, $f':X'\to Z'$, $v:Z'\to Z$, as in the statement of Theorem \ref{Main}, such that $f':X'\to Z'$ is birational over $Z'$ to an Abelian group scheme $f":X"\to Z'$. One then gets from Lemma \ref{Phi} that $X'$ is isomorphic to $X"$, and thus that $f':X'\to Z'$ is itself an Abelian group scheme. \end{proof}





\begin{proof} (of Theorem\ref{main}). That II implies I and III is obvious. 

We first prove that I implies II. Since $X$ contains no rational curve, $K_X$ is nef, hence semi-ample by Conjecture \ref{abund}. Let $f:X\to Z$ be its Moishezon-Iitaka fibration. Its smooth fibres have torsion canonical bundle and do not contain rational curves. Applying Theorem \ref{Main}, we may replace $f:X\to Z$ by a finite \'etale cover $X'$, still weakly special, and $f$ by $f':X'\to Z'$ which is an Abelian group scheme. We have for the Kodaira dimensions: $dim(Z')=dim(Z)=\kappa(X)=\kappa(X')$. We thus get a holomorphic moduli map $\mu:Z\longrightarrow \mathcal{A}_{g,P}$ to the moduli space of Abelian varieties of dimension $g:=\dim(X)-\dim(Z)$ with a suitable polarization. The image is a compact subvariety of general type (this moduli space carries a Hermitian metric with Griffiths semi-negative curvature and negative holomorphic sectional curvature, see \cite[Lemma 5.9.3]{Kol}).  
If $X$, and so $Z,X',Z'$ are weakly special, the moduli map $\mu$ is constant and $f'$ is isotrivial;  hence $\dim(Z)=\kappa(X)=\kappa(Z)$.  Therefore $Z'$ is of general type and special.  Thus $\dim(Z')=0$, and $X$ is the (only) fiber of $f$, hence an Abelian variety (after \'etale cover).

We next prove that III implies II. If $X$ contains a Zariski dense entire curve, so do $Z'$ and $\mu(Z')$, in the notations above. Since $Z'$ is of general type, it must have dimension zero by Lang's conjecture [L]. The conclusion is the same as in the previous proof.

\end{proof}

\section{Proof of Theorem~\ref{hyp}.}

\begin{proof}[Proof of Theorem~\ref{hyp}] 
Proceed by contradiction and suppose that $f: \mathbf{C} \longrightarrow X$ is an entire curve.   Let $V : = \overline{f(\mathbf{C})}^{\text{Zar}}$ denote its Zariski closure.  We choose the entire curve such that the dimension $d \geq 1$ of $V$ is minimal.  Since we assumed that $X$ does not support any rational curves or images of Abelian varieties, $d >1$.  Let $V^{\text{norm}}$ denote the normalization of $V$.  Since $V$ is not uniruled,  the Kodaira dimension of $V^{\text{norm}}$ is nonnegative, by Conjecture~\ref{abund}.  

The singularities of $V$ can be resolved to produce a smooth model $V^{\text{smth}}$. Running the minimal model program then yields a terminal good minimal model $V^{\text{term}}$ of $V^{\text{smth}}$.  The canonical bundle of the terminal model is nef,  and hence,  semi-ample by the abundance conjecture.  Let $\Phi : V^{\text{term}} \longrightarrow W$ denote the Moishezon--Iitaka map.  The generic fiber $V_w^{\text{term}} = \Phi^{-1}(w)$ has terminal singularities and vanishing Kodaira dimension.  Hence, by the abundance conjecture,  the canonical bundle of $V_w^{\text{term}}$ is torsion.  Since $V$ does not contain any rational curve, $V_w^{\text{term}}$ does not contain a Zariski dense set of rational curves.  Hence,  by Conjecture~\ref{CY+},  it is a terminal quotient of an Abelian variety.  In particular,  $V_w^{\text{term}}$ contains Zariski dense entire curves.   Since the dimension $d$ is chosen to be minimal,  we must have $V_w^{\text{term}} \simeq V^{\text{term}}$, and therefore, $W$ is a point, and $V^{\text{norm}}$ is a terminal quotient of a simple Abelian variety.  

From Lemma \ref{Phi}, item 3, we get: $V^{\text{norm}} \simeq V^{\text{term}}$, and so $V^{norm}$ is a terminal undercover of an Abelian variety, contradicting our assumptions. Thus $X$ is (Brody) hyperbolic. 

We finally show that $K_X$ is ample. Since $X$ does not contain rational curves, $K_X$ is nef (by \cite{Mori}), hence semi-ample (by Conjecture \ref{abund}). If $\Phi:X\to Z$ is the Moishezon-Iitaka fibration, its smooth fibres have torsion canonical bundle and no rational curve, hence are \'etale undercovers of Abelian varieties. Contradiction unless $X=Z$, which means that $K_X$ is ample, as claimed.
\end{proof}

\begin{remark}\label{rem cy} 
In the proof of Theorem~\ref{hyp}, we only made use of the following weaker form of Conjecture~\ref{CY+}: Let $V^{\text{norm}}$ be a normal projective variety with no rational curves.  Let $V^{\text{term}}$ be a terminal model with torsion canonical bundle,  and write $\psi : V^{\text{term}} \dashrightarrow V^{\text{norm}}$ for the rational map.  Let $\text{Exc}(\psi)$ and $\Gamma(\psi)$ denote the exceptional locus and graph of $\psi$, respectively.  If all rational curves of $V^{\text{term}}$ are contained in the image of $\text{Exc}(\psi)$ under the projection $ \Gamma(\psi) \longrightarrow V^{\text{term}}$,  then $V^{\text{term}}$ is a terminal quotient of an Abelian variety.

\end{remark}

\section{Examples: The Proof of Theorems ~\ref{ex} and \ref{ex'}.}
\noindent In this section,  we prove Theorems~\ref{ex} and \ref{ex'},  producing examples of canonically polarised manifolds with exceptional locus (in the sense of Lang) any given simple Abelian variety, or more generally manifolds with Zariski dense entire curves.

\subsection*{The Construction}
\noindent Let $V \subset \mathbf{P}^n$ be a projective manifold later chosen to be a simple Abelian variety of dimension $d$, or any projective manifold containing a Zariski dense entire curve.  Let $W_0 \subset \mathbf{P}^{n+1}$ be the cone over $V$ with vertex $v$.  Let $\beta : P \longrightarrow \mathbf{P}^{n+1}$ be the blow-up of $v$ in $\mathbf{P}^{n+1}$.  Write $\widetilde{W_0}$ for the strict transform of $W_0$ in $P$.    We note that $\pi : W'_0 \longrightarrow W_0$ is the blow-up $v$ in $W_0$,  with exceptional divisor $\pi^{-1}(v) \simeq V$.  

We have a natural inclusion of $P$ in $\mathbf{P}^{n+1}\times \mathbf{P}^n$.  Let $\lambda: P\longrightarrow \mathbf{P}^n$ be the second projection, inducing the linear projection of $\mathbf{P}^{n+1}$ to $\mathbf{P}^n$ (undetermined at $v$).  The restriction $\lambda': W_0'\longrightarrow V$ of $\lambda$ to $W_0'$ is,  therefore,  a $\mathbf{P}^1$-bundle.

Let now $\psi:A\longrightarrow \mathbf{P}^{n+1}$ be a surjective, finite map from a simple, $(n+1)$-dimensional Abelian variety $A$ onto $\mathbf{P}^{n+1}$. We assume that $\psi$ is unramified over $v$, vertex of $W_0$.

Let $A':=A\times_{\mathbf{P}^{n+1}}P$ with its two projections $\alpha:A'\longrightarrow A$ and $\pi:A'\longrightarrow P$.  The map $\alpha$ is the blow-up of $A$ at the finitely many points of $\psi^{-1}(v)$.  Let $X_0:=\psi^{-1}(W_0)\subset A$, and $X$ be the strict transform of $X_0$ in $A'$.  This is merely the blow-up of $X_0$ at $\psi^{-1}(v)$.  By a Bertini-type theorem, we can move $W_0$ by a generic $g\in \text{Aut}(\mathbf{P}^{n+1})$  in order to make $X_0$ smooth outside of $\psi^{-1}(v)$, or equivalently,  to ensure that $X$ is smooth. 

The proof of Theorem \ref{ex} follows from the next: 

\begin{proposition}\label{lem ex} Choose for $V$ a simple Abelian variety.
The projective manifold $X$ constructed above is canonically polarised, does not contain any rational curve and any elliptic curve if $d\geq 2$. Any entire curve in $X$ is Zariski dense in some of the exceptional divisors $V_j\cong V$ of $\alpha:X\longrightarrow X_0$ which lie over $\psi^{-1}(v)$ and are the only subvarieties of $X$ not of general type. Moreover, the cotangent bundle of $X$ is nef and big, and the Kobayashi pseudodistance of $X$ is a (non-degenerate) distance on $X\setminus W$, and vanishes on $W$, the union of the $V_j's$. 
\end{proposition}

\begin{proof} Since $A$ is simple, any of its subvarieties, and so in particular $X_0$, are of general type. So is thus the manifold $X$. Since neither $A$ nor $V$ carries rational curves, $K_X$ is ample (nef by \cite{Mori}, ample by \cite{kaw'}). Since $A$ is simple, $X_0$ does not carry any rational curve or Abelian subvariety. Thus $X_0$ is hyperbolic, and any entire curve of $X$ must be contained in the exceptional locus of $\alpha:X\longrightarrow X_0$, which is the union of the $V_j's$. Since each $V_j$ is isomorphic to the simple Abelian variety $V$, each entire curve of $X$ is contained and Zariski dense in one of the $V_j's$ (by Bloch's theorem).

We prove that the cotangent bundle of $X$ is nef and big. By \cite{D}, Proposition 5, it is sufficient for this to check that $X\subset B=A\times V$ is non-degenerate, that is, that the natural difference map $\delta:X\times X\to B$ defined by  $\delta(x,y):=x-y$ is generically finite, using the group law on $B$, whatever the origin. Since $A$ and $V$ are simple and not isogeneous (they have different dimensions), the only Abelian subvarieties of $B$ are $A$ and $V$, and so  we only need to check that for both projections from $B$ to $A$ and $V$, the image of $X$ is either everything, or of the same dimension as $X$. But this is obvious, since $X$ maps surjectively onto $V$, and birationally to $X_0$ in $A$. 

We now prove that the Kobayashi pseudodistance on $X$ is a distance on $X\setminus W$. To this end, we observe that $X_0\subset A$ is Brody hyperbolic, and therefore, by the theorems of Brody \cite[3.6.3]{KobayashiHyperbolicComplexSpaces} and Bloch \cite[3.9.19]{KobayashiHyperbolicComplexSpaces}, it follows that $X_0$ is Kobayashi hyperbolic. Let $d_{\bullet}$ denote the Kobayashi pseudodistance. The decreasing property of $d_{\bullet}$ shows that $d_X\geq \alpha^*(d_{X_0})$. On the other hand, $d_{X}$ vanishes on $W$, and so $d_X=\alpha^*(d_{X_0})$, which is a metric on $X_0$. This can be seen by lifting chains of discs in $X_0$ to $X$ and using the fact that $d_V$ is identically zero, so that liftings of discs abutting to distinct points in $V$ can be connected inside $V$ by adding discs, without increasing the pseudodistance. We leave the easy details to the reader. 
\end{proof}

\begin{remark} If a projective manifold (Calabi--Yau or hyperk\"ahler) $V$ exists which is not hyperbolic, but contains no rational curve and no image of an Abelian variety, the preceding construction gives a general type manifold $X$ embedded in $A\times V$ which is not hyperbolic, but violates Lang conjecture (the conditional Theorem \ref{hyp}). 
\end{remark}

\begin{proof}[Proof of Theorem~\ref{ex'}.] Choose $V$ to be a projective manifold containing a Zariski dense entire curve in the construction of $X\subset B:=A\times V$ in the proof of Theorem \ref{ex}. The arguments given there apply (except that now $V$ may contain rational curves, so that $X$ is only of general type, and does not need to have a nef and big cotangent bundle). Indeed, $X\to X_0\subset A$ is a blow-up in finitely many points and a smooth model of $X_0$. Since $A$ is simple, $X_0$ and its smooth models are of general type, and do not contain any entire curve (by Bloch's theorem again). By brody theorem, $d_{X_0}$ is a metric, and $d_{X}\geq \alpha^*(d_{X_0})$. But we do not know whether we have equality. At least, $d_X$ is a distance on $X\setminus W$. \end{proof}

\begin{remark}\label{d} 
If $d_V$ vanishes identically, then we still have, as above, the equality $d_{X}= \alpha^*(d_{X_0})$. If $V$ is `special' in the sense of \cite{Ca04}, it is conjectured that $d_V$ vanishes identically. Hence, it is conjectured that this equality holds. If $V$ contains a dense (for the metric topology) entire curve, then $d_V$ vanishes, and $d_{X}= \alpha^*(d_{X_0})$.
\end{remark}

\begin{question} Do the $X's$ constructed from $V$ and $A$ above have a big cotangent bundle? We leave this question open.
\end{question}

\begin{remark}\label{exc set'} Let us consider the product $X\times Z$ where $X\subset A\times V$ is one of the examples constructed in the proof of Theorems \ref{ex} or \ref{ex'}, and $Z$ is canonically polarised and hyperbolic. The exceptional set (in the sense of S. Lang) of $X\times Z$ is then a finite number of copies of $V\times Z$.
\end{remark}

\section{Remarks on metrics with quasi-negative holomorphic sectional curvature}

\noindent Recall that the Kobayashi--Lang conjecture predicts that a compact hyperbolic K\"ahler manifold is canonically polarised, which is the case if it admits a K\"ahler metric with negative holomorphic sectional curvature (\cite{WuYau,TosattiYang}). This remains true for particular classes of pluriclosed manifolds (see \cite{BroderStanfield,LeeStreets}). 

It is well known that canonically polarised manifolds need not be hyperbolic: Fermat hypersurfaces in $\mathbf{P}^n$ of degree $d>n+1$ support rational curves.  Hirzebruch \cite{Hirzebruch} and,  more recently,  Sarem \cite{Sarem} (building on \cite{HS}),  have produced canonically polarised manifolds with no rational curves.  These examples are toroidal compactifications of ball quotients;  hence, they admit elliptic curves and are not hyperbolic. 

\begin{remark}\label{rem ex} The referee showed us easier examples of non-hyperbolic surfaces with no rational curves, but with ample canonical bundle: the second symmetric power of a bi-elliptic, non-hyperelliptic curve $C$ of genus at least $3$. The exceptional locus is then a finite union of elliptic curves.
\end{remark}

The examples constructed in Theorem~\ref{ex} however show that canonically polarised manifolds may have arbitrary simple Abelian variety as their exceptional loci.

The examples constructed by Hirzebruch \cite{Hirzebruch} and Sarem \cite{Sarem} are also of interest since they provide examples of compact K\"ahler manifolds with (a K\"ahler metric of) quasi-negative holomorphic sectional curvature\footnote{The holomorphic sectional curvature $\text{HSC} : T_X^{1,0} \longrightarrow \mathbf{R}$ is said to be \textit{quasi-negative}  if $\text{HSC}(\xi_x) \leq 0$ for all $\xi_x \in T_{X,x}^{1,0}$ and all $x \in X$,  and there exists some $x \in X$ such that $\max_{\xi_x \in T_{X,x}^{1,0} \setminus \{ 0 \}} \text{HSC}(\xi_x) <0$.},  but no (Hermitian) metric of negative holomorphic sectional curvature.  Since these examples are compactifications of ball quotients $\mathbf{B}^n / \Gamma$,  it is immediate that the (holomorphic sectional) curvature is negative on $\mathbf{B}^n/\Gamma$.  Hence,  the difficulty in establishing the quasi-negative curvature lies in showing that the curvature is globally non-positive.

\begin{proposition}\label{prop:ex+curv}
The projective manifolds $X$ constructed in Theorem~\ref{ex} have K\"ahler metrics of nonpositive bisectional curvature and pseudoconvex Finsler metrics of quasi-negative holomorphic sectional curvature.  If $\dim X =2$,  then $X$ has no K\"ahler metric of nonpositive sectional curvature.  

\end{proposition} \begin{proof}
Maintaining the previous notation,  let us recall that $X$ is a submanifold of $A \times V$,  where $A$ and $V$ are Abelian varieties.  We also constructed $X$ as the blow-up of $X_0 : = \psi^{-1}(W) \subset A$ at $\psi^{-1}(v)$.  We will write $\pi$ for the blow-up map.  Let $\delta$ be the metric on $X$ induced from a flat K\"ahler metric on $A \times V$.  From the curvature decreasing property of the holomorphic bisectional curvature,  it follows immediately that $X$ has nonpositive holomorphic bisectional curvature.  

Let $X_0^{\circ}$ denote the regular locus of $X_0$.  We note that $X_0$ has ample cotangent bundle,  and Kobayashi's construction \cite{KobayashiFinsler} yields a pseudoconvex Finsler metric with negative curvature on $X_0$.  We leave the reader to consult \cite{KobayashiFinsler} for the definitions,  where they are well presented.  We note, however,  that what we call `pseudoconvex',  Kobayashi refers to as `convex' in \cite{KobayashiFinsler}.

Let $U : = \pi^{-1}(X_0^{\circ}) \subseteq X$.  Let $\eta \in \mathcal{C}_0^{\infty}(U)$ be a smooth cut-off function,  compactly supported in $U$.  Let $F$ denote the (pullback to $U$) of the negatively curved pseudoconvex Finsler metric on $X_0^{\circ}$.  In particular,  the holomorphic sectional curvature (interpreted in the sense of Wu \cite{Wu}) of $\eta F$ is negative on the support of $\eta$.   We then apply Wu's theorem \cite{Wu} on the holomorphic sectional curvature of the sum of (Finsler) metrics to deduce that the pseudoconvex Finsler metric $\delta + \eta F$ has quasi-negative holomorphic sectional curvature.

The last assertion follows from a result of Zheng \cite{Zheng} which states (in particular) that a canonically polarised surface with a K\"ahler metric of nonpositive sectional curvature is hyperbolic. 
\end{proof}

From Proposition~\ref{prop:ex+curv} (and motivated by questions in \cite{DiverioSurvey,DT}) we naturally propose the following question: 

\begin{question}\label{QN}
Do the projective manifolds $X$ constructed in Theorem~\ref{ex} support K\"ahler or Hermitian metrics of quasi-negative holomorphic sectional curvature?
\end{question}

The toroidal compactifications of ball quotients considered in \cite{Hirzebruch, Sarem} have quasi-negative holomorphic bisectional curvature,  but no Hermitian metrics of negative holomorphic sectional curvature (since they are not hyperbolic!).  In this case,  the difficulty concentrates in showing that the negatively curved metrics on the ball quotients retain globally nonpositive curvature on the compactification (c.f., \cite{HS}).  The difficulty in our situation is reversed: The examples constructed in Theorem~\ref{ex} are submanifolds of Abelian varieties and thus,  have globally nonpositive holomorphic bisectional curvature.  Hence, the problem is concentrated on producing strict negativity at some point.

\end{document}